\documentclass[reqno]{amsart}
\usepackage{amscd}
\usepackage{epsfig}  
\usepackage[utf8]{inputenc}
\usepackage{latexsym,amssymb}
\usepackage{amsfonts}
\usepackage[T1]{fontenc}
\usepackage{amsthm}
\usepackage{amsmath,multicol,multirow,eso-pic}
\usepackage{graphicx}
\usepackage{color}
\usepackage{arcs}
\usepackage{float}
\usepackage{tikz}
\usepackage{tkz-tab}
\usepackage{array}
\usetikzlibrary{intersections,arrows,decorations.pathmorphing,decorations.pathreplacing,decorations.text,decorations.fractals,fit,backgrounds,positioning,petri ,calc,through,shapes.misc,chains,scopes,mindmap,calendar,trees,shadows,patterns,shapes,decorations,external,folding,shadings}

\usepackage{pgfplots}

\usepackage[all,knot]{xy}

\theoremstyle{plain}
\newtheorem{thm}{Theorem}[section]

\newtheorem{lem}[thm]{Lemma}
\newtheorem{prop}[thm]{Proposition}
\theoremstyle{definition}
\newtheorem{defn}{Definition}[section]

\newtheorem{rem}{Remark}[section]

\numberwithin{equation}{section}

\makeatletter
\newcounter{knot@strings}
\newif\ifknot@draftmode
\tikzoption{save knot path}{\tikz@addmode{\pgfsyssoftpath@getcurrentpath\knot@tmppath\expandafter\global\expandafter\let#1=\knot@tmppath}}
\tikzoption{use knot path}{\tikz@addmode{\expandafter\pgfsyssoftpath@setcurrentpath#1}}
\tikzset{
  knot/draft mode/.is if=knot@draftmode
}
\newenvironment{knot}[1][]{%
  \tikzset{#1}%
  \setcounter{knot@strings}{0}}{%
  \foreach \knot@str in {1,...,\the\value{knot@strings}} {
    \expandafter\expandafter\expandafter\draw\expandafter\expandafter\expandafter[\csname knot@string@opts@\knot@str\endcsname,use knot path=\csname knot@string@\knot@str\endcsname] (0,0);
    \ifknot@draftmode
    \begingroup
    \let\pgfsyssoftpath@movetotoken=\pgfqpoint
    \let\pgfsyssoftpath@linetotoken=\pgfqpoint
    \let\pgfsyssoftpath@curvetotoken=\pgfqpoint
    \let\pgfsyssoftpath@curvetosupportatoken=\pgfqpoint
    \let\pgfsyssoftpath@curvetosupportbtoken=\pgfqpoint
    \csname knot@string@\knot@str\endcsname
    \global\pgf@xa=\pgf@x
    \global\pgf@ya=\pgf@y
    \endgroup
    \node[circle,fill=white,fill opacity=.5] at (\pgf@xa,\pgf@ya) {\knot@str};
\fi
  }
  \pgfmathtruncatemacro{\knot@stam}{\the\value{knot@strings}-1}
  \foreach \knot@sta in {1,...,\knot@stam} {
    \pgfmathtruncatemacro{\knot@stap}{\knot@sta + 1}
    \foreach \knot@stb in {\knot@stap,...,\the\value{knot@strings}} {
      \pgfintersectionofpaths{\expandafter\pgfsetpath\csname knot@string@\knot@sta\endcsname}{\expandafter\pgfsetpath\csname knot@string@\knot@stb\endcsname}
    \foreach \intsect in {1,...,\pgfintersectionsolutions} {
      \pgfpointintersectionsolution{\intsect}
      \pgfgetlastxy{\intsectx}{\intsecty}
        \ifknot@draftmode
        \node[circle,fill=white,fill opacity=.5] at (\intsectx,\intsecty)         {\knot@sta-\knot@stb-\intsect};
\else
\@ifundefined{knot@crossing@\knot@sta-\knot@stb-\intsect}{
}{
\pgfmathtruncatemacro{\knot@under}{\csname knot@crossing@\knot@sta-\knot@stb-\intsect\endcsname == \knot@sta ? \knot@stb : \knot@sta}
\expandafter\let\expandafter\knot@over\csname knot@crossing@\knot@sta-\knot@stb-\intsect\endcsname
\pgfscope
\clip (\intsectx,\intsecty) circle[radius=10pt];
    \expandafter\expandafter\expandafter\draw\expandafter\expandafter\expandafter[\csname knot@string@opts@\knot@under\endcsname,use knot path=\csname knot@string@\knot@under\endcsname] (0,0);
    \expandafter\expandafter\expandafter\draw\expandafter\expandafter\expandafter[\csname knot@string@opts@\knot@over\endcsname,use knot path=\csname knot@string@\knot@over\endcsname,white,line width=3\pgflinewidth] (0,0);
    \expandafter\expandafter\expandafter\draw\expandafter\expandafter\expandafter[\csname knot@string@opts@\knot@over\endcsname,use knot path=\csname knot@string@\knot@over\endcsname] (0,0);
\endpgfscope
}
\fi
      }
    }
  }
}
\newcommand{\strand}[1][]{%
  \stepcounter{knot@strings}%
  \expandafter\def\csname knot@string@opts@\the\value{knot@strings}\endcsname{#1}%
\path[save knot path=\csname knot@string@\the\value{knot@strings}\endcsname]}
\newcommand{\crossing}[2]{%
\expandafter\def\csname knot@crossing@#1\endcsname{#2}}

\makeatother

\begin{document}

\title[]{$B_3$ BLOCK REPRESENTATIONS OF DIMENSION 6 AND BRAID REVERSIONS}

\author{Taher I. Mayassi \and Mohammad N. Abdulrahim }

\address{Taher I. Mayassi\\
         Department of Mathematics and Computer Science\\
         Beirut Arab University\\
        P.O. Box 11-5020, Beirut, Lebanon}
\email{tim187@student.bau.edu.lb}

\address{Mohammad N. Abdulrahim\\
         Department of Mathematics and Computer Science\\
         Beirut Arab University\\
         P.O. Box 11-5020, Beirut, Lebanon}
\email{mna@bau.edu.lb}

\begin{abstract}
We construct a family of six dimensional block representations of the braid group $B_3$ on three strings. 
We show that some of these representations can be used to separate braids from their reversed braids of some known knots and others of 9 and 10 crossings.

\end{abstract}

\maketitle

\medskip

\renewcommand{\thefootnote}{}
\vskip 0.1in 

\section{Introduction}
\bigskip
Gauss was the first mathematician who studied knots mathematically in the 1800s.
Reidemeister and Alexander (around 1930), were able to make significant progress in knot
theory, which  has been a very dynamic branch of topology especially after the discovery
of the Jones polynomial in 1984 and its connections with quantum field theory, as well as
some concrete applications in the study of enzymes acting on DNA strands  \cite{Mur}.\\
The reverse of an oriented knot $K$ is defined as the same knot with the opposite
orientation. Vertibility seems to be very difficult to detect.
The connection between knot theory and braid theory was discovered in
1923 by Alexander. He proved that every oriented knot or link is isotopic to a closed
braid \cite{Ban}.\\
In \cite{Lie}, Lieven Le Bruyn introduced some simple representations of the braid group $B_3$ that are able to separate the braids of the following knots from their reversed braids, $6_3$, $7_5$, $8_7$, $8_9$, $8_{10}$ and $8_{17}$ which is the first
non-invertible kont with minimal number of crossings. All these knots have at most 8 crossings and
are closures of 3-string braids. 
The braid $b=\sigma_1^{-2}\sigma_2\sigma_1^{-1}\sigma_2\sigma_1^{-1}\sigma_2^2$
is the braid whose closure is the knot $8_{17}$, and the braid $b'=\sigma_2^2\sigma_1^{-1}\sigma_2\sigma_1^{-1}\sigma_2\sigma_1^{-2}$
is the reversed braid of $b$.
It turns out  that the trace of the braid $b$ is different from the trace of the reversed braid $b'$ 
for sufficiently large $B_3$-representations. Bruce Westbury discovered 12-dimensional 
representations of $B_3$ that are able to detect a braid from its reversed braid \cite{We}.\\
In section 2,  we state some essential definitions and theorems.
In section 3, we present some basic results about detecting vertibility and separating braids of some 
knots from their reversed braids using representations of $B_3$, the braid group on three strings.
In \cite{Lie}, Lieven Le Bruyn was able to detect reversions 
using simple representations of $B_3$. In fact, the author in \cite{Lie} discovered a 6-dimensional 
representation of $B_3$ which separates a braid of the knot $8_{17}$ from its reversed braid. The 
knots inspected by Lieven were $6_3,7_5,8_7,8_9,8_{10}$, and $8_{17}$ which is the non-invertible knot with minimal number of crossings.
In section 4, we construct a family of block representations of $B_3$ of dimension 6. 
In section 5, we prove that some complex specializations of the block representations, constructed in 
section 4, are able to separate braids from their reversed braids of the above knots, in addition to 
some knots of 9 and 10 crossings like $9_{6}$, $9_{9}$ and $10_{5}$.
Note that, even if 3-braid $\alpha$ is not conjugate to the reversed braid $\alpha'$, this does not 
mean that the closure of $\alpha$ is non-invertible. In fact, $8_{17}$ is the only non-invertible
knot with minimal number of crossings, and of which the braid coming from the knot is separated
from its reversed braid. In order to apply Theorem 2.2, section 2, to show 
that the closure of 3-braid $\alpha$ is non-invertible, one needs to show (i) that $\alpha$ is not 
conjugate to $\alpha'$ and (ii) that $\alpha$ is not conjugate to certain type of braids (flypes).
In this paper, section 5, we work on the separation of braids from their reversed braids and
we construct a table of some knots with their braids listed there.
In addition to separating braids with at most 8 crossings from their reversed braids, we extend such results to include some knots of 9 and 10 crossings as $9_6$, $9_9$ and $10_5$.

\bigskip
\section{Definitions and theorems}
\medskip
\begin{defn} \cite{Bir}
The braid group on $n$ strings,  $B_n$, is the abstract group with the presentation
$$B_n=\langle \sigma_1,\cdots,\sigma_{n-1} | \sigma_i\sigma_j=\sigma_j\sigma_i\text{ for }|i-j|>1\text{ and }\sigma_i\sigma_{i+1}\sigma_i=\sigma_{i+1}\sigma_{i}\sigma_{i+1}\\
\text{ for }i=1,\cdots,n-2\rangle.$$
\end{defn}

\begin{defn}\cite{Cro}
A knot $K$ is the image of a homeomorphism of a unit circle $S^1$ into $\mathbb{R}^3$ considered up to continuous deformations (ambient isotopies) in the following sense.\\
Two knots $K_1$ and $K_2$ are equivalent (isotopic) if there exists a continuous mapping $H:\mathbb{R}^3\times[0,1]\to\mathbb{R}^3$ such that
\begin{enumerate}
\item For every $t\in[0,1]$ the mapping $x\mapsto H(x,t)$ is a homeomorphism of $\mathbb{R}^3$ onto $\mathbb{R}^3$
\item $H(x,0)=x$ for all $x\in\mathbb{R}^3$.
\item $H(K_1,1)=K_2$ 
\end{enumerate} 
Such mapping $H$ is called ambient isotopy.
\end{defn} 

\begin{defn}
A link is a finite union of pairwise disjoint knots, which are called the components of the link.
\end{defn} 

A closed braid is a braid in which the corresponding ends of its strings are connected in pairs. This means that every braid can be closed up to be a knot or a link. Now, we have the following theorem.
\begin{thm}Alexander's Theorem\cite{Ban} 
Every knot or link can be represented as a closed braid.
\end{thm}
\begin{figure}[h]

\hfil\epsfig{file=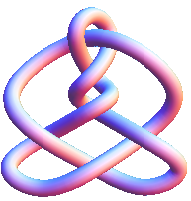, width=1.25in}\hfil
\caption{\label{fig} Knot $7_5$}
\vspace{.7cm}
\resizebox{10cm}{2cm}{
\begin{tikzpicture}

\begin{knot}
\strand[red,ultra thick]  (0,0) .. controls +(1,0) and +(-1,0) .. ++(2,0) ..controls +(1,0) and +(-1,0) .. ++(2,0) ..controls +(1,0) and +(-1,0) .. ++(2,0) ..controls +(1,0) and +(-1,0) .. ++(2,0) .. controls +(1,0) and +(-1,0) .. ++(2,-1) .. controls +(1,0) and +(-1,0) .. ++(2,-1) .. controls +(1,0) and +(-1,0) .. ++(2,0).. controls +(1,0) and +(-1,0) .. ++(2,0);

\strand[blue,ultra thick] (0,-1) .. controls +(1,0) and +(-1,0) .. ++(2,-1) .. controls +(1,0) and +(-1,0) .. ++(2,1).. controls +(1,0) and +(-1,0) .. ++(2,-1).. controls +(1,0) and +(-1,0) .. ++(2,1).. controls +(1,0) and +(-1,0) .. ++(2,1).. controls +(1,0) and +(-1,0) .. ++(2,0).. controls +(1,0) and +(-1,0) .. ++(2,-1) .. controls +(1,0) and +(-1,0) .. ++(2,1);

\strand[violet,ultra thick] (0,-2) .. controls +(1,0) and +(-1,0) .. ++(2,1) .. controls +(1,0) and +(-1,0) .. ++(2,-1).. controls +(1,0) and +(-1,0) .. ++(2,1).. controls +(1,0) and +(-1,0) .. ++(2,-1).. controls +(1,0) and +(-1,0) .. ++(2,0).. controls +(1,0) and +(-1,0) .. ++(2,1).. controls +(1,0) and +(-1,0) .. ++(2,1).. controls +(1,0) and +(-1,0) .. ++(2,-1);

\crossing{2-3-1}{2}
\crossing{2-3-2}{3}
\crossing{2-3-3}{2}
\crossing{2-3-4}{3}
\crossing{1-2-1}{1}
\crossing{1-3-1}{3}
\crossing{2-3-5}{2}
\crossing{2-3-6}{3}
\end{knot}
\end{tikzpicture}
}
\caption{Representative braid of $7_5$: $\sigma_1^{4}\sigma_2\sigma_1^{-1}\sigma_2^2$ }

\end{figure}
Every knot or Link may be closure of many braids even with different number of strings. However, Markov's theorem gives necessary and sufficient conditions for the closures of two braids to give equivalent knots or links \cite{Ban}. One of the sufficient conditions is conjugation. That is, if two braids are conjugate then their closures are equivalent links. For example, the braids $\sigma_1^{-1}\sigma_2\sigma_1^{-3}\sigma_2^3$ and $\sigma_2^3\sigma_1^{-1}\sigma_2\sigma_1^{-3}$ are associated with the same knot $8_9$.

\begin{defn}\cite{Fra}
The minimal number of strings needed in braid to represent a knot or link $K$ is called the braid index of $K$.
\end{defn}
\begin{defn}
The reverse of a braid of the form $\sigma_1^{n_1}\sigma_2^{m_1}\sigma_1^{n_2}\sigma_2^{m_2}\cdots\sigma_1^{n_k}\sigma_2^{m_k}$ is the braid $\sigma_2^{m_k}\sigma_1^{n_k}\cdots\sigma_2^{m_2}\sigma_1^{n_2}\sigma_2^{m_1}\sigma_1^{n_1},$ where $n_1,m_1,\cdots, n_k,m_k$ are integers.
\end{defn}
\begin{defn}\cite{Cro}
A knot is said to be invertible if it can be deformed continuously to itself, but with the orientation reversed.
\end{defn}
Before stating the next theorem, we need these definitions. 
\begin{defn}\cite{BirM}
A knot of braid index 3 is said to admit a flype  if its associated braids are conjugate to a braid of the form$$\sigma_1^a\sigma_2^b\sigma_1^c\sigma_2^{\epsilon}$$ for some integers $a,b,c$, $\epsilon=\pm1$.
\end{defn}

\begin{figure}[h]
\resizebox{10cm}{2cm}{
\begin{tikzpicture}

\begin{knot}
\strand[red,ultra thick]  (0,0) .. controls +(1,0) and +(-1,0) .. ++(2,0) ..controls +(1,0) and +(-1,0) .. ++(2,-1) ..controls +(1,0) and +(-1,0) .. ++(2,1) ..controls +(1,0) and +(-1,0) .. ++(2,0) .. controls +(1,0) and +(-1,0) .. ++(2,0) .. controls +(1,0) and +(-1,0) .. ++(2,-1);

\strand[blue,ultra thick] (0,-1) .. controls +(1,0) and +(-1,0) .. ++(2,-1) .. controls +(1,0) and +(-1,0) .. ++(2,0).. controls +(1,0) and +(-1,0) .. ++(2,0).. controls +(1,0) and +(-1,0) .. ++(2,1).. controls +(1,0) and +(-1,0) .. ++(2,-1).. controls +(1,0) and +(-1,0) .. ++(2,0);

\strand[violet,ultra thick] (0,-2) .. controls +(1,0) and +(-1,0) .. ++(2,1) .. controls +(1,0) and +(-1,0) .. ++(2,1).. controls +(1,0) and +(-1,0) .. ++(2,-1).. controls +(1,0) and +(-1,0) .. ++(2,-1).. controls +(1,0) and +(-1,0) .. ++(2,1).. controls +(1,0) and +(-1,0) .. ++(2,1);

\crossing{1-3-1}{1}
\crossing{1-3-2}{3}
\crossing{1-3-3}{1}
\crossing{2-3-1}{3}
\crossing{2-3-2}{2}
\crossing{2-3-3}{3}

\end{knot}

\end{tikzpicture}
}
\caption{Representative braid of $6_3$: $\sigma_1^{-1}\sigma_2^2\sigma_1^{-2}\sigma_2$}
\end{figure}
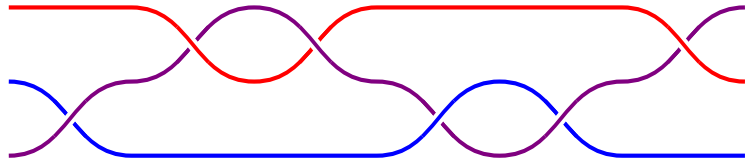

\begin{defn}
A flype is said to be non-degenerate when its associated braid $\sigma_1^a\sigma_2^b\sigma_1^c\sigma_2^{\epsilon}$ and its reverse $\sigma_2^{\epsilon}\sigma_1^c\sigma_2^b\sigma_1^a$ are in distinct conjugacy classes. 
\end{defn}

\begin{thm}\cite{BirM}\label{th1}
Let $\mathcal{K}$ be a link of braid index 3 with oriented 3-braid representative $K$. Then $\mathcal{K}$ is non-invertible if and only if $K$ and its reverse braid $K'$  are in distinct conjugacy classes, and $K$ does not contain a representative which admits a non-degenerate flype.
\end{thm} 
\bigskip 

\section{Basic results}
\medskip

There is an infinite family of non-invertible knots, \cite{Tro}. The knot $8_{17}$, which is the closure of the braid $\sigma_1^{-1}\sigma_2\sigma_1^{-1}\sigma_2^2\sigma_1^{-2}\sigma_2$, is the unique non-invertible knot with a minimal number of crossings. The following table gives the numbers of non-invertible and invertible knots according to their number of crossings up to 16 \cite{Mur}.

\begin{center}
\begin{tabular}{|>{\centering\arraybackslash}p{1.5cm}|> {}c|c|c|c|c|c|c|c|c|c|c|c|c|c|}
\hline 
Number of crossings & $3$ & $4$ & $5$ & $6$ & $7$ & $8$ & $9$ & $10$ & $11$ & $12$ & $13$ & $14$ & $15$ & $16$\\ 
\hline 
Non-invertible & $0$ & $0$ & $0$ & $0$ & $0$ & $1$ & $2$ & $33$ & $187$ & $1144$ & $6919$ & $38118$ & $226581$ &$1309875$ \\ 
\hline 
Invertible & $1$ & $1$ & $2$ & $3$ & $7$ & $20$ & $47$ & $132$ & $365$ & $1032$ & $3069$ & $8854$ & $267121$ &$78830$\\ 
\hline 
\end{tabular} 
\end{center}
\vspace{.5cm}
Notice that some non-invertible knots, referred to in the table above, are: $8_{17}$, $9_{32}$, $9_{33}$, $10_{67}$, $10_{80}$, $10_{81}$, $10_{83}$, for more details see \cite{Mur}.\\
Imre Tuba and Hans Wenzl introduced a complete classification of all simple $B_3$-representations of dimension $\leq5$ \cite{Tub}. We easily check that none of these representations can detect invertibility.\\
Bruce Westbury found a representation of dimension 12 that is able to detect a braid from its reversed braid by taking traces. The question, which was raised after that, was about determining the minimal dimension of a $B_3$-representation which detects knot vertibility.\\
Lieven Le Bruyn proposed a general method to solve the separation problems for three string braids
\cite{Lie}. In fact, he succeeded to solve Westbury’s separation problem using simple representations of $B_3$ of dimension 6. A specific representation of $B_3$ is given by the matrices
  
$$\sigma_1= \begin{pmatrix}
p+1&p-1&p-1&p-1&-p+1&-p+1\\
-2p-1&-1&-2p-1&2p+1&-2p-1&2p+1\\
p+2&p+2&-p&-p-2&-p-2&p+2\\
-p-2&-3p&p+2&-p+2&3p&-p-2\\
p-1&-p+1&3p+3&-p+1&3p+1&-3p-3\\
-3&-2p-1&2p+1&3&2p+1&-2p-3\\
\end{pmatrix},$$\\
$$\sigma_2= \begin{pmatrix}
p+1&p-1&p-1&-p+1&p-1&p-1\\
-2p-1&-1&-2p-1&-2p-1&2p+1&-2p-1\\
p+2&p+2&-p&p+2&p+2&-p-2\\
p+2&3p&-p-2&-p+2&3p&-p-2\\
-p+1&p-1&-3p-3&-p+1&3p+1&-3p-3\\
3&2p+1&-2p-1&3&2p+1&-2p-3\\
\end{pmatrix},$$\\

where $p$ is a primitive third root of unity.
This representation seems to be able to separate a braid of the knot $8_{17}$ from its reversed braid \cite{Lie}.\\

Lieven Le Bruyn constructed Zariski dense family of simple $B_3$-representations, which are able to detect vertibility of knots, having at most 8 crossings, and which are closures of 3-string braids. The knots inspected by Lieven were $6_3,7_5,8_7,8_9,8_{10}$ (which are 'flypes')
and $8_{17}$ which is the non-invertible knot with minimal number of crossings.
\bigskip

\section{Constructing block representations of $B_3$ of dimension 6}
\medskip
In this section, we construct representations of the braid group $B_3$ of dimension six. Let $A$, $B$, $C$ and $D$ be $3\times3$ non-zero matrices. Let $\rho$ be a mapping from $B_3$ to $M_6(\mathbb{C})$, the vector space of $6\times6$ matrices over the complex vector space $\mathbb{C}$. 
This mapping is given by 
$$\rho(\sigma_1)=\left(\begin{matrix}
A&B\\
C&D\\
\end{matrix}\right)\quad\text{and}\quad
\rho(\sigma_2)=\left(\begin{matrix}
A&-B\\
-C&D\\
\end{matrix}\right),$$\\\\
where $\sigma_1$ and $\sigma_2$ are the generators of $B_3$.
\begin{prop}
The mapping $\rho:B_3\to GL(6,\mathbb{C})$ defines a representation of the braid group $B_3$ if and only if $\det(\rho(\sigma_i))\neq0$ $(i=1,2)$ and the matrices $A$, $B$, $C$ and $D$ satisfy the following relations.
\begin{itemize}
\item[] \begin{equation}\label{E2}
 A^2B-BCB-ABD+BD^2=0
\end{equation}
\item[] \begin{equation}\label{E3}
 CA^2-DCA-CBC+D^2C=0
\end{equation} 
\end{itemize}
\end{prop}

\begin{proof}
Recall that the generators $\sigma_1$ and $\sigma_2$ of the braid group $B_3$ satisfy the relation $\sigma_1\sigma_2\sigma_1=\sigma_2\sigma_1\sigma_2$. Then $\rho(\sigma_1\sigma_2\sigma_1)=\rho(\sigma_2\sigma_1\sigma_2)$. This  implies that\\
$$\left(\begin{matrix}
A&B\\
C&D\\
\end{matrix}\right)
\left(\begin{matrix}
A&-B\\
-C&D\\
\end{matrix}\right)
\left(\begin{matrix}
A&B\\
C&D\\
\end{matrix}\right)
=\left(\begin{matrix}
A&-B\\
-C&D\\
\end{matrix}\right)
\left(\begin{matrix}
A&B\\
C&D\\
\end{matrix}\right)
\left(\begin{matrix}
A&-B\\
-C&D\\
\end{matrix}\right).$$\\
Therefore, $$\begin{pmatrix}
A^3-BCA-ABC+BDC&A^2B-BCB-ABD+BD^2\\\\
CA^2-DCA-CBC+D^2C&CAB-DCB-CBD+D^3\\
\end{pmatrix}$$\\\\
$$=\begin{pmatrix}
A^3-BCA-ABC+BDC&-A^2B+BCB+ABD-BD^2\\\\
-CA^2+DCA+CBC-D^2C&CAB-DCB-CBD+D^3\\
\end{pmatrix}.$$\\\\
So, $A^2B-BCB-ABD+BD^2=0$ \;\; and \;\; $CA^2-DCA-CBC+D^2C=0$
\end{proof}
\bigskip

\section{Knots and detecting inversion among their braids}
\medskip
In this section, we separate braids from their reversed braids of some knots using the representations constructed in the previous section.  
\begin{lem}\label{lambda}
Given  $I_3$  the identity matrix of dimension 3. Suppose that $D=\lambda I_3$ for some $\lambda\in\mathbb{C}\setminus\{0\}$, which is not an eigenvalue of $A$. Let $B$ be an invertible matrix. Then $\rho$ is a representation of $B_3$ if and only if  $C=B^{-1}A^2-\lambda B^{-1}A+\lambda^2B^{-1}$.  
\end{lem}
\begin{proof}
If $D=\lambda I_3$ and $B$ is invertible then, by direct computation, the equations \ref{E2} and \ref{E3} imply that $C=B^{-1}A^2-\lambda B^{-1}A+\lambda^2B^{-1}$. Now,  $\det(\rho(\sigma_1))=\det(\rho(\sigma_2))=\det(AD-BD^{-1}CD)=\det(-A^2+2\lambda A-\lambda^2 I_3)=\det(-(A-\lambda I_3)^2)=-\det((A-\lambda I_3))^2\neq0$. 
\end{proof}
Attempts to separate conjugacy classes of braids for many knots from their reversed braids, using the representation defined by Lemma \ref{lambda} have not been successful so far. We then make other speculations for the matrices $A$, $B$ and $D$. We take the matrices  $A$, $B$ and $D$ as follows:
$$A=\left(\begin{matrix}
a&a-2&a-2\\
-2a+1&d&-2a+1\\
f&f&g\\
\end{matrix}\right),\;\;
B=\left(\begin{matrix}
a-2&-a+2&-a+2\\
2a-1&-2a+1&2a-1\\
-f&-f&f\\
\end{matrix}\right)$$
 and
$$D=\left(\begin{matrix}
-a+3&3a-3&-a-1\\
-a+2&3a-2&-3a\\
3&2a-1&-2a-1\\
\end{matrix}\right),$$
where $a,d,f$ and $g$ are complex numbers.
\newline

The matrices $\rho(\sigma_1)$ and $\rho(\sigma_2)$ are
$$\rho(\sigma_1)=\left(\begin{matrix}
A&B\\
C&D\\
\end{matrix}\right)
\text{ and }
\rho(\sigma_2)=\left(\begin{matrix}
A&-B\\
-C&D\\
\end{matrix}\right),$$
where the matrix $C$ satisfies the equations \ref{E2} and \ref{E3}. The complex numbers $a,d,f,g$ are chosen to have the matrices of $\sigma_1$ and $\sigma_2$ invertible. In order to solve for the matrix $C$ in the equation \ref{E2}, we require that $B$ is an invertible matrix\\\\
In the next proposition, we give values to $d,g,f$, all in terms of $a$; which guarantees the invertibility of the  matrix $B$. This will be done in a way that the map $\rho$ is a representation of the braid group $B_3$ in $GL(6,\mathbb{C})$.

\begin{prop}\label{P}
$\rho$ is a representation of $B_3$ if $a\not\in\{-1,0,2,\frac{1}{2}\}$ and either one of the following conditions holds true.
\begin{enumerate}
\item[1)] $d=1+2a,\; g=-1\pm ia\sqrt{3},\; f\neq0$\\
\item[2)] $d=1+2a,\; g=3a-1,\; f=1+a$\\
\item[3)] $(d=1-a + ia\sqrt{3},\; g=-1 - ia\sqrt{3},\; f\neq0)$ or $(d=1-a - ia\sqrt{3},\; g=-1 + ia\sqrt{3},\; f\neq0)$\\
\item[4)] $(d=1-a + ia\sqrt{3},\; g=-1 + ia\sqrt{3},\; f=1+a)$ or $(d=1-a - ia\sqrt{3},\; g=-1 - ia\sqrt{3},\; f=1+a)$\\
\item[5)] $d=1-a\pm ia\sqrt{3},\; g=3a-1,\; f=1+a$.\\
\end{enumerate}
\end{prop}
\begin{proof}
The determinant of $B$ is $4(a-2)(2a-1)f\neq0$. So $B$ is invertible.
Substituting $C=B^{-1}A^2-B^{-1}ABDB+D^2B^{-1}$ in the equation \ref{E3}, we get 9 equations with 4 unknown complex numbers $a,d,f,g$. We fix $a$ and we solve for the numbers $d,f,g$. Using Mathematica software, we solve the system of 9 equations to get the solutions stated above.\\
The invertibility of the matrices $\rho(\sigma_1)$ and $\rho(\sigma_2)$ comes from the fact that the
determinant of $\rho(\sigma_i)$ is $-64a^6$ once we assign values to $d,g$ and $f$ as in the second
condition of the Proposition \ref{P}. Also, the determinant of $\rho(\sigma_i)$ ($i=1,2$) is $32(1\pm i\sqrt{3})a^6$ under the conditions 1, 3, 4, 5 of Proposition \ref{P}.
Since $a\neq0$, it follows that $\det(\rho(\sigma_i))\neq0$ ($i=1,2$). Therefore $\rho(\sigma_1)$ and $\rho(\sigma_2)$ are invertible.  
\end{proof}
\begin{rem}
Let $p$ be a primitive third root of unity. By taking $a=p+1$ and $f=p+2$ in condition 3 of Proposition \ref{P}, we get the representation in section 3, which Lieven Le Bruyn used to separate the braid describing the knot $8_{17}$ from its reveresd braid.
\end{rem} 
\begin{thm}
A family of representations $\rho$ is able to separate braids from their reversed braid on the list
of several knots which are the closures of three string braids, and which some of them have more than 
8 crossings.
\end{thm}
\begin{proof}
Consider the representation $\rho $ of $B_3$ given in Proposition~\ref{P} with the condition $d=1-a-ia\sqrt{3}, g=-1+ia\sqrt{3}, f\neq0$.
The representation $\rho$ becomes\\\\
$$\rho(\sigma_1)=\begin{pmatrix}
a&a-2&a-2&a-2&-a+2&-a+2\\
-2a+1&1-a-ia\sqrt{3}&-2a+1&2a-1&-2a+1&2a-1\\
f&f&-1+ia\sqrt{3}&-f&-f&f\\
-1-a&1+(-2+i\sqrt{3})a&1+a &3-a&-3+3a&-1-a\\
\frac{2-2a-a^2}{-2+a}& ia\sqrt{3}&\frac{[(3+i\sqrt{3})(1+a)+(3-i\sqrt{3})f]a}{2f}&-a+2&3a-2&-3a\\
\frac{4-a-2a^2}{-2+a}&-1+a(-1+i\sqrt{3})&\frac{2+[1-i\sqrt{3}+(3-i\sqrt{3})f]a-(1+i\sqrt{3})a^2}{2f}
&3&2a-1&-2a-1
\end{pmatrix}$$\\
and\\
$$\rho(\sigma_2)=\begin{pmatrix}
a&a-2&a-2&-a+2&a-2&a-2\\
-2a+1&1-a-ia\sqrt{3}&-2a+1&-2a+1&2a-1&-2a+1\\
f&f&-1+ia\sqrt{3}&f&f&-f\\
1+a&-1-(-2+i\sqrt{3})a&-1-a &3-a&-3+3a&-1-a\\
\frac{-2+2a+a^2}{-2+a}& -ia\sqrt{3}&-\frac{[(3+i\sqrt{3})(1+a)+(3-i\sqrt{3})f]a}{2f}&-a+2&3a-2&-3a\\
\frac{-4+a+2a^2}{-2+a}&1-a(-1+i\sqrt{3})&-\frac{2+[1-i\sqrt{3}+(3-i\sqrt{3})f]a-(1+i\sqrt{3})a^2}{2f}
&3&2a-1&-2a-1
\end{pmatrix}$$\\\\
Next, we take three different values of $a$ and $f$. Thus we obtain three different 
representations of $B_3$. We show that these representations are able to separate the braids of the 
following knots: $6_3,7_5,8_7,8_9, 8_{10},8_{17}, 9_6,9_9,10_5$ from their reversed braids. The 
author in \cite{Lie} succeeded to separate some knots up to 8 crossings.
Our representations were able to recognize knots with 9 and 10 crossings that are separated from their reversed braids. More precisely, we perform the following calculations as shown in the table below.
\\\\
\renewcommand{\arraystretch}{1.25}
\begin{tabular}{|c|>{\centering\arraybackslash}p{3cm}|>{\centering\arraybackslash}p{3.5cm}|>{\centering\arraybackslash}p{3.5cm}|>{\centering\arraybackslash}p{3.5cm}|}
\hline
\multirow{2}{*}{Knot}&\multirow{2}{*}{Braid word $w$}&\multicolumn{3}{c|}{$Tr(\rho(w))-Tr(\rho(w'))$}\\
\cline{3-5}
 &  & $a=2-3i,f=7.3$  &  $a=1.5+i,\; f=6-4.2i$  & $a=1+3i,\;f=10.2+10.3i$  \\
 \hline 
$6_3$ & $\sigma_1^{-1}\sigma_2^2\sigma_1^{-2}\sigma_2$ & $170.17+47.15i$ & $201.38-11.75i$ & $427.9+123.1i$  \\ 
\hline 
$7_5$ & $\sigma_1^{4}\sigma_2\sigma_1^{-1}\sigma_2^2$ & $1.96\times10^7+1.52\times10^7i$ & $-419.54-142.8i$ & $2.3\times10^6+2.8\times10^7i$  \\ 
\hline 
$8_7$ & $\sigma_1^{4}\sigma_2^{-2}\sigma_1\sigma_2^{-1}$ & $3624.8+23139i$ & $-9244-3706.1i$ & $73,847.2-58,855.3i$  \\ 
\hline 
$8_9$ & $\sigma_1^{-1}\sigma_2\sigma_1^{-3}\sigma_2^3$ & $170.17+47.15i$ & $201.38-11.75i$ & $427.9+123.1i$ \\ 
\hline 
$8_{10}$ & $\sigma_1^{-1}\sigma_2^2\sigma_1^{-2}\sigma_2^3$ & $3624.8+23139i$ & $-9244-3706.1i$ & $73,847.2-58,855.3i$ \\ 
\hline 
$8_{17}$ & $\sigma_1^{-1}\sigma_2^2\sigma_1^{-1}\sigma_2^2\sigma_1^{-2}\sigma_2$ & $-524.9-581.1i$ & $-459+182.3i$& $-510.4-653.8i$ \\
\hline 
$9_6$ & $\sigma_1^2\sigma_2^2\sigma_1^5\sigma_2^{-1}$ & $7.5\times10^8-3.2\times10^9i$ & $-1.6\times10^7-1.5\times10^7i$ & $-5.2\times10^9-3.1\times10^9i$ \\ 
\hline 
$9_9$ & $\sigma_1^3\sigma_2^{-1}\sigma_1^4\sigma_2^2$ & $7.5\times10^8-3.2\times10^9i$ & $-1.6\times10^7-1.5\times10^7i$ & $-5.2\times10^9-3.1\times10^9i$ \\ 
\hline 
$10_5$ & $\sigma_1^{-2}\sigma_2\sigma_1^{-1}\sigma_2^6$ & $-3.5\times10^6+2.6\times10^6i$ & $308,285+628,064i$ & $-1.7\times10^6-2.9\times10^7i$ \\ 
\hline 
\end{tabular}\\
 
Here, $w'$ denotes the reverse of the braid $w$.
\end{proof} 

\bigskip

\textbf{Conflict of Interest.} On behalf of all authors, the corresponding author states that there is no conflict of interest.

\end{document}